\documentclass[final,intlimits]{article}

\usepackage{mathtools}
\usepackage{amssymb}
\usepackage{amsthm}
\usepackage{xcolor}
\usepackage{tikz}

\DeclareSymbolFont{bbold}{U}{bbold}{m}{n}
\DeclareSymbolFontAlphabet{\mathbbold}{bbold}

\newcommand{\N}{\mathbb{N}}
\newcommand{\Z}{\mathbb{Z}}
\newcommand{\Q}{\mathbb{Q}}
\newcommand{\R}{\mathbb{R}}
\newcommand{\C}{\mathbb{C}}

\newcommand{\T}{\mathbb{T}}
\newcommand{\1}{\mathbbold{1}}

\newcommand{\loc}{\mathrm{loc}}
\newcommand{\lu}{\mathrm{unif}}
\newcommand{\unif}{\mathrm{unif}}
\newcommand{\Mloc}{\mathcal{M}_{\loc}(\R)}
\newcommand{\M}{\mathcal{M}_{\loc,\lu}(\R)}

\newcommand{\Id}{I}

\newcommand{\Dir}{{\rm D}}
\newcommand{\Neu}{{\rm N}}
\newcommand{\SL}{\mathrm{SL}}

\newcommand{\dist}{\mathrm{dist}}

\makeatletter
\newcommand{\uD}{\u@DN{\Dir}}
\newcommand{\uN}{\u@DN{\Neu}}
\newcommand{\u@DN}[1]{u_{#1\mkern-1mu}}
\makeatother

\DeclareMathOperator{\spt}{spt}
\DeclareMathOperator{\diam}{diam}
\DeclareMathOperator{\sgn}{sgn}
\renewcommand{\Re}{\operatorname{Re}}

\DeclareMathOperator{\tr}{tr}

\newcommand{\from}{\colon}

\let\phi\varphi
\let\le\leqslant
\let\leq\leqslant
\let\ge\geqslant
\let\geq\geqslant

\makeatletter
\def\@row#1,{#1\@ifnextchar;{\@gobble}{&\@row}}
\def\@matrix{%
    \expandafter\@row\my@arg,;%
    \@ifnextchar({\\ \get@in@paren{\@matrix}}{\after@matrix}%
    }
\def\matrixtype#1#2#3{%
    \ifmmode\def\after@matrix{\end{#2}\right#3}%
    \else\def\after@matrix{\end{#2}\right#3$}$\fi
    \left#1\begin{#2}\get@in@paren{\@matrix}%
    }
\def\@column#1,{#1\@ifnextchar;{\@gobble}{\\ \@column}}
\newcommand\vect{}
\def\svect(#1){\left(\begin{smallmatrix}\@column#1,;\end{smallmatrix}\right)}
\def\vect{\get@in@paren{\@vect}}
\def\@vect{\left(\begin{matrix}\expandafter\@column\my@arg,;\end{matrix}\right)}
\def\get@in@paren#1({\def\my@arg{}\def\my@rest{}\def\after@get{#1}\get@arg}
\let\e@a\expandafter
\def\get@arg#1){\e@a\kl@test\my@rest#1(;}
\def\kl@test#1(#2;{\e@a\def\e@a\my@arg\e@a{\my@arg#1}%
                   \ifx:#2:\let\my@exec\after@get
                   \else\let\my@exec\get@arg
                        \e@a\def\e@a\my@arg\e@a{\my@arg(}%
                        \def@rest#2;%
                   \fi\my@exec}
\def\def@rest#1(;{\def\my@rest{#1\kl@zu}}
\def\kl@zu{)}

\makeatletter
\newcommand\MyPairedDelimiter{%
  \@ifstar{\My@Paired@Delimiter{{}}}
          {\My@Paired@Delimiter{}}%
}
\newcommand\My@Paired@Delimiter[4]{%
  \newcommand#2{%
    \@ifstar{\start@PD{#1}{\delimitershortfall=-1sp}{#3}{#4}}
            {\start@PD{#1}{}{#3}{#4}}%
  }%
}
\newcommand\start@PD[5]{%
  #1\mathopen{\mathpalette\put@delim@helper{\put@delim{#2}{#3}{.}{#5}}}%
  #5%
  \mathclose{\mathpalette\put@delim@helper{\put@delim{#2}{.}{#4}{#5}}}%
}
\newcommand\put@delim@helper[2]{%
  \hbox{$\m@th\nulldelimiterspace=0pt #2#1$}%
}
\newcommand\put@delim[5]{%
  \setbox\z@\hbox{$\m@th#5{#4}$}%
  \setbox\tw@\null
  \ht\tw@\ht\z@ \dp\tw@\dp\z@
  #1#5%
  \left#2\box\tw@\right#3%
}

\makeatother
\MyPairedDelimiter*{\abs}{\lvert}{\rvert}
\MyPairedDelimiter*{\norm}{\lVert}{\rVert}
\MyPairedDelimiter{\set}{\{}{\}}

\newcommand\llim{
\mathchoice{\vcenter{\hbox{${\scriptstyle{-}}$}}}
{\vcenter{\hbox{$\scriptstyle{-}$}}}
{\vcenter{\hbox{$\scriptscriptstyle{-}$}}}
{\vcenter{\hbox{$\scriptscriptstyle{-}$}}}}
\newcommand\rlim{
\mathchoice{\vcenter{\hbox{${\scriptstyle{+}}$}}}
{\vcenter{\hbox{$\scriptstyle{+}$}}}
{\vcenter{\hbox{$\scriptscriptstyle{+}$}}}
{\vcenter{\hbox{$\scriptscriptstyle{+}$}}}}

\theoremstyle{plain} 
\newtheorem{theorem}{Theorem}[section]
\newtheorem{corollary}[theorem]{Corollary}
\newtheorem{lemma}[theorem]{Lemma}
\newtheorem{proposition}[theorem]{Proposition}
\theoremstyle{definition}
\newtheorem{example}[theorem]{Example}
\newtheorem*{definition}{Definition}
\newtheorem{remark}[theorem]{Remark}

\usepackage{enumitem}

\setenumerate[1]{nolistsep} 
\setenumerate[2]{nolistsep} 

\newcommand{\Hmm}[1]{\leavevmode{\marginpar{\tiny%
$\hbox to 0mm{\hspace*{-0.5mm}$\leftarrow$\hss}%
\vcenter{\vrule depth 0.1mm height 0.1mm width \the\marginparwidth}%
\hbox to 0mm{\hss$\rightarrow$\hspace*{-0.5mm}}$\\\relax\raggedright #1}}}

\begin{document}

\medmuskip=4mu plus 2mu minus 3mu
\thickmuskip=5mu plus 3mu minus 1mu
\belowdisplayshortskip=9pt plus 3pt minus 5pt

\title{A quantitative version of Gordon's Theorem for Jacobi and Sturm-Liouville operators}

\author{Christian Seifert}

\date{\today}

\maketitle

\begin{abstract}
  We prove a quantitative version of Gordon's Theorem concerning absence of eigenvalues for Jacobi matrices and Sturm-Liouville operators with complex coefficients.

  \bigskip

  MSC2010: 34L15, 34L40, 81Q12

  \medskip

  Key words: Jacobi matrices, Sturm-Liouville operators, eigenvalue problem, quasiperiodic operators
\end{abstract}

\section{Introduction}
\label{sec:intro}

In this paper we study the absence of eigenvalues for (discrete) Jacobi operators
\[(H_{a,b}u)(n) := a(n+1)u(n+1) + b(n)u(n) + a(n) u(n-1) \quad(n\in\Z)\]
where $a,b\in\ell_\infty(\Z)$ such that $\frac{1}{a(\cdot)}\in \ell_\infty(\Z)$,
and analogously (continuum) Sturm-Liouville operators
\[H_{a,\mu}u := -\partial a\partial u + u\mu,\]
where $a\in L_\infty(\R)$ such that $\frac{1}{a}\in L_\infty(\R)$ and $\mu$ is locally a complex Radon measure (for precise definitions on $\mu$ see Section \ref{sec:SL} below).

Under the assumption that the coefficients of the operators can be (locally) approximated by periodic ones in a suitable sense we prove absence of eigenvalues for these operators.
While in the Jacobi case we will work with $\ell_\infty$-approximation, for the Sturm-Liouville case we
use approximation in $L_1(\R)$ for $a$ and a weak Wasserstein-type metric for $\mu$.
Controlling the approximation rate in a quantitative way we obtain lower bounds on the modulus of eigenvalues, where the bound is determined by the (norms of the) coefficents of the operator and the approximation rate.

In this way, our result may be seen as a quantitative version of Gordon's Theorem, which first appeared in \cite{Gordon1976}, and since then was subsequently generalized \cite{Gordon1986,Damanik2000,DamanikStolz2000, Damanik2004, Seifert2011, Seifert2012, Krueger2013, SeifertVogt2014}.
However, all the previous results stick to the case of (discrete or continuum) Schr\"odinger operators, and except of \cite{SeifertVogt2014} to a qualitative statement in the self-adjoint case.
The first quantitative result appeared in \cite{SeifertVogt2014} for not necessarily self-adjoint continuum Schr\"odinger operators on $L_2(\R)$, where the eigenvalue bound is proven to be sharp.

We will not only consider the operators in $\ell_2(\Z)$ and $L_2(\R)$, respectively, but in $c_0(\Z)$ and $C_0(\R)$, so that we in turn obtain eigenvalue bounds also for the whole $\ell_p(\Z)$ and $L_p(\R)$ scale for $1\leq p<\infty$.
Note that all the sequence and function spaces we will work with are complex-valued, so that the coefficients may be complex, thus obtaining (in general) non-self-adjoint operators.

The paper is organised as follows. In Section \ref{sec:Jacobi} we deal with the discrete Jacobi case, where we also provide an example for quasiperiodic coefficients.
Section \ref{sec:SL} is devoted to the continuum Sturm-Liouville case. Here, we will also comment on optimality of our eigenvalue bound. In the appendix we provide discrete and continuum versions of Gronwall's lemma and a short lemma on norms of $SL(2,\C)$-matrices.

\section{The Jacobi case}
\label{sec:Jacobi}

For sequences $a,b\in\ell_\infty(\Z)$ such that $\frac{1}{a(\cdot)}\in \ell_\infty(\Z)$ we consider the Jacobi matrix $H_{a,b}\from c_0(\Z)\to c_0(\Z)$, where $c_0(\Z)$ is the space of complex sequences with index set $\Z$ converging to $0$ at $\pm\infty$, defined by
\[(H_{a,b}u)(n) := a(n+1)u(n+1) + b(n)u(n) + a(n) u(n-1) \quad(n\in\Z).\]
We are going to prove the following theorem.

\begin{theorem}
\label{thm:Jacobi}
  Let $a,b\in\ell_\infty(\Z)$ such that $\frac{1}{a(\cdot)}\in \ell_\infty(\Z)$.
  Assume there exists a sequence $(p_m)$ in $\N$ with $p_m\to \infty$ and $C>0$ such that
  \begin{equation}
    \label{eq:Jacobi_exp}
    e^{Cp_m} \max_{k=-p_m+1,\ldots,p_m} \bigl(\abs{a(k+1) - a(k+1+p_m)}+ \abs{b(k) - b(k+p_m)}\bigr)\to 0.  
  \end{equation}  
  Then $H_{a,b}$ does not have any eigenvalues with modulus less than $\frac{e^{C/2}-\norm{a}_\infty}{\norm{\frac{1}{a}}_\infty} + \norm{b}_\infty-1$.
\end{theorem}

The condition \eqref{eq:Jacobi_exp} in the preceding theorem states that the difference of the three pieces 
$\bigl((a(n+1),b(n))\bigr)_{n\in\set{-p_m+1,\ldots,0}}$, $\bigl((a(n+1),b(n))\bigr)_{n\in\set{1,\ldots,p_m}}$ and $\bigl((a(n+1),b(n))\bigr)_{n\in\set{p_m+1,\ldots,2p_m}}$
tends to zero (as $m\to \infty$) faster than the exponential $e^{-Cp_m}$. This may be seen as a quantitative (in the sense of the exponential rate) version of Gordon's condition for potentials of discrete Schr\"odinger operators, see e.g.\ \cite{Gordon1976}.
In case condition \eqref{eq:Jacobi_exp} holds true for all $C>0$, we obtain an analogue of Gordon's theorem \cite{Gordon1976} for Jacobi matices.

\begin{corollary}
  Let $a,b\in\ell_\infty(\Z)$ such that $\frac{1}{a(\cdot)}\in \ell_\infty(\Z)$.
  Assume that condition \eqref{eq:Jacobi_exp} holds true for all $C>0$. Then $H_{a,b}$ does not have any eigenvalues.
\end{corollary}

Note that since $\ell_p(\Z)\subseteq c_0(\Z)$ for all $1\leq p<\infty$, we obtain the same result for $H_{a,b}$ considered as an operator in $\ell_p(\Z)$ with $1\leq p<\infty$.

\begin{corollary}
   Let $1\leq p<\infty$, $a,b\in\ell_\infty(\Z)$ such that $\frac{1}{a(\cdot)}\in \ell_\infty(\Z)$.
   Assume that \eqref{eq:Jacobi_exp} is satisfied. Then $H_{a,b}\from 
\ell_p(\Z)\to\ell_p(\Z)$ does not have any eigenvalues with modulus less 
than $\frac{e^{C/2}-\norm{a}_\infty}{\norm{\frac{1}{a}}_\infty} + 
\norm{b}_\infty-1$.
   In case \eqref{eq:Jacobi_exp} holds true for all $C>0$, then $H_{a,b}$ does not have any eigenvalues.
\end{corollary}

\begin{remark}
  Such a result cannot hold in $\ell_\infty(\Z)$, since periodic Jacobi matrices (i.e., Jacobi matrices with periodic sequences $a$ and $b$) have periodic eigensolutions (which are therefore in $\ell_\infty(\Z)$).
\end{remark}

We will split the proof into several lemmas.

\begin{lemma}
  Let $a,b\in\ell_\infty(\Z)$, $a(n)\neq 0$ for all $n\in\Z$, $z\in\C$ and 
$u\from \Z\to \C$. The following are equivalent:
  \begin{enumerate}
    \item
    $u$ is a solution of the difference equation
    \[a(n+1) u(n+1) + b(n) u(n) + a(n) u(n-1) - z u(n) = 0 \quad(n\in\Z).\]
    \item
    For $m,n\in\Z$ we have
    \[\begin{pmatrix} u(m+1)\\ a(m+1)u(m)\end{pmatrix} = T_z(m,n)\begin{pmatrix} u(n+1)\\ a(n+1)u(n)\end{pmatrix},\]
    where
    \[T_z(m,n) = \begin{cases}
		  M_z(m-1)\cdots M_z(n+1)M_z(n) & m>n,\\
		  \Id & m=n,\\
		  M_z(m)^{-1}\cdots M_z(n-2)^{-1}M_z(n-1)^{-1} & m<n,
                 \end{cases}\]
    and
    \[M_z(n) = \begin{pmatrix}
      \frac{z-b(n+1)}{a(n+2)} & -\frac{1}{a(n+2)} \\
      a(n+2) & 0
    \end{pmatrix}.\]                 
  \end{enumerate}
\end{lemma}

\begin{proof}
  ``(a) $\Rightarrow$ (b)'':
  Fix $m,n\in\Z$ and let
  \[T_z(m,n)\from \begin{pmatrix} u(n+1)\\ a(n+1)u(n)\end{pmatrix} \to \begin{pmatrix} u(m+1)\\ a(m+1)u(m)\end{pmatrix}.\]
  Then $T_z(m,n)$ is linear and thus can be represented by a matrix (which we will also denote by $T_z(m,n)$).
  Since $u$ satisfies the difference equation, we compute
  \[\begin{pmatrix} 
      u(n+2) \\ 
      a(n+2)u(n+1)
    \end{pmatrix} = 
    \underbrace{\begin{pmatrix}
      \frac{z-b(n+1)}{a(n+2)} & -\frac{1}{a(n+2)} \\
      a(n+2) & 0
    \end{pmatrix}}_{=M_z(n)}
    \begin{pmatrix} 
      u(n+1)\\
      a(n+1)u(n)
     \end{pmatrix}.\]
    Thus, 
    \[T_z(m,n) = \begin{cases}
		  M_z(m-1)\cdots M_z(n+1)M_z(n) & m>n,\\
		  \Id & m=n,\\
		  M_z(m)^{-1}\cdots M_z(n-2)^{-1}M_z(n-1)^{-1} & m<n.
                 \end{cases}\]

  ``(b) $\Rightarrow$ (a)'':
  A direct computation yields 
  \begin{align*}
    & T_z(n+1,n) \begin{pmatrix} u(n+1)\\ a(n+1)u(n)\end{pmatrix} \\
    & = \begin{pmatrix} \frac{1}{a(n+2)} \bigl(zu(n+1) - b(n+1)u(n+1) - a(n+1)u(n)\bigr) \\ a(n+1)u(n) \end{pmatrix}.
  \end{align*}
  Thus,
  \[u(n+2) = \frac{1}{a(n+2)} \bigl(zu(n+1) - b(n+1)u(n+1) - a(n+1)u(n)\bigr),\]
  which means that $u$ satisfies the difference equation.  
\end{proof}

Note that $\det T_z(m,n) = 1$ for all $m,n\in\Z$. Furthermore, $T_z(m,n)$ ``depends locally'' on $a$ and $b$, i.e.\ $T_z(m,n)$ depends only on $a(k+1)$ and $b(k)$ for $k\in\{\min\{m,n\}+1,\ldots,\max\{m,n\}\}$. 

\begin{lemma}
\label{lem:Jacobi_per}
  Let $a,b\in\ell_\infty(\Z)$ be $p$-periodic, $a(n)\neq 0$ for all $n\in\Z$, $z\in\C$, $u$ a solution of the difference equation
  \[a(n+1) u(n+1) + b(n) u(n) + a(n) u(n-1) - z u(n) = 0 \quad(n\in\Z).\]
  Then
  \[\max\set{\norm{\begin{pmatrix}
		    u(n+1) \\ a(n+1) u(n)
                   \end{pmatrix}};\; n\in\set{-p,p,2p}} \geq \frac{1}{2} \norm{\begin{pmatrix}
		    u(1) \\ a(1) u(0)
                   \end{pmatrix}}.\] 
\end{lemma}

\begin{proof}
  By periodicity of $a$ and $b$ we have $T_z(p,0) = T_z(2p,p) = T_z(0,-p) =: T$.
  The Cayley-Hamilton Theorem yields (since $\det T = 1$)
  \[T^2  - (\tr T)T + \Id = 0.\]
  In case $\abs{\tr T}\leq 1$, applying this equation to $(u(1),a(1)u(1))^\top$ we observe
  \[\begin{pmatrix} u(2p+1) \\ a(2p+1) u(2p) \end{pmatrix} - \tr(T) \begin{pmatrix} u(p+1) \\ a(p+1) u(p) \end{pmatrix} = - \begin{pmatrix} u(1) \\ a(1) u(0) \end{pmatrix},\]
  and therefore
  \[\max\set{\norm{\begin{pmatrix}
		    u(n+1) \\ a(n+1) u(n)
                   \end{pmatrix}};\; n\in\set{p,2p}} \geq \frac{1}{2} \norm{\begin{pmatrix}
		    u(1) \\ a(1) u(0)
                   \end{pmatrix}}.\] 
  If $\abs{\tr(T)}>1$ we apply the equation to $(u(-p+1),a(-p+1)u(-p))^\top$ to get
  \[\begin{pmatrix} u(p+1) \\ a(p+1) u(p) \end{pmatrix} + \begin{pmatrix} u(-p+1) \\ a(-p+1) u(-p) \end{pmatrix} = \tr(T) \begin{pmatrix} u(1) \\ a(1) u(0) \end{pmatrix}.\]
  Thus,
  \[\max\set{\norm{\begin{pmatrix}
		    u(n+1) \\ a(n+1) u(n)
                   \end{pmatrix}};\; n\in\set{-p,p}} \geq \frac{1}{2} \norm{\begin{pmatrix}
		    u(1) \\ a(1) u(0)
                   \end{pmatrix}}. \qedhere\]  
\end{proof}

\begin{lemma}
\label{lem:Jacobi_diff}
  Let $a,\tilde{a},b,\tilde{b}\in\ell_\infty(\Z)$, $a(n),\tilde{a}(n)\neq 0$ for all $n\in\Z$, $z\in\C$, $u,\tilde{u}$ solutions of the difference equations
  \begin{align*}
    a(n+1) u(n+1) + b(n) u(n) + a(n) u(n-1) - z u(n) & = 0 \quad(n\in\Z),\\
    \tilde{a}(n+1) \tilde{u}(n+1) + \tilde{b}(n) \tilde{u}(n) + \tilde{a}(n) \tilde{u}(n-1) - z \tilde{u}(n) & = 0 \quad(n\in\Z),
  \end{align*}
  respectively, satisfying
  \[\begin{pmatrix}
      u(1)\\a(1)u(0)
    \end{pmatrix} = \begin{pmatrix}
      \tilde{u}(1)\\\tilde{a}(1)\tilde{u}(0)
    \end{pmatrix}.\]
  Then, for $n\in\Z$, we have
  \begin{align*}
    & \norm{\begin{pmatrix}
      u(n+1) - \tilde{u}(n+1)\\a(n+1)u(n) - \tilde{a}(n+1)\tilde{u}(n)
    \end{pmatrix}}\\
    & \leq \prod_{k=\min\{n+1,1\}}^{\max\{0,n\}} \norm{A(k)} \sum_{k=\min\{n+1,1\}}^{\max\{0,n\}} \norm{B(k)} \norm{\begin{pmatrix} \tilde{u}(k+\frac{1-\sgn n}{2}) \\ \tilde{a}(k+\frac{1-\sgn n}{2})\tilde{u}(k-\frac{1+\sgn n}{2})\end{pmatrix}},
    \end{align*}
    where
    \begin{align*}
      A(n) & = \begin{pmatrix}
      \frac{z-b(n)}{a(n+1)} & - \frac{1}{a(n+1)} \\
      a(n+1) & 0
    \end{pmatrix},\\
     B(n) & = \begin{pmatrix} 
	\left(\frac{z-b(n)}{a(n+1)} - \frac{z-\tilde{b}(n)}{\tilde{a}(n+1)}\right) & \left(- \frac{1}{a(n+1)} +\frac{1}{\tilde{a}(n+1)}\right) \\
	a(n+1) - \tilde{a}(n+1) & 0
      \end{pmatrix} \quad(n\in\Z).
     \end{align*}
\end{lemma}

\begin{proof}
  First, assume that $n> 0$.
  We have
  \begin{align*}
    & \begin{pmatrix}
      u(n+1) - \tilde{u}(n+1)\\a(n+1)u(n) - \tilde{a}(n+1)\tilde{u}(n)
    \end{pmatrix} \\
    & = \begin{pmatrix}
      \frac{z-b(n)}{a(n+1)} u(n) - \frac{1}{a(n+1)} a(n) u(n-1) - \frac{z-\tilde{b}(n)}{\tilde{a}(n+1)} \tilde{u}(n) + \frac{1}{\tilde{a}(n+1)} \tilde{a}(n) \tilde{u}(n-1) \\
      a(n+1) u(n) - \tilde{a}(n+1) \tilde{u}(n) 
    \end{pmatrix} \\
    & = \underbrace{\begin{pmatrix}
      \frac{z-b(n)}{a(n+1)} & - \frac{1}{a(n+1)} \\
      a(n+1) & 0
    \end{pmatrix}}_{=A(n)}     
    \begin{pmatrix}
      u(n) - \tilde{u}(n)\\a(n)u(n-1) - \tilde{a}(n)\tilde{u}(n-1)
    \end{pmatrix} \\&\quad
    + \underbrace{\begin{pmatrix} 
	\left(\frac{z-b(n)}{a(n+1)} - \frac{z-\tilde{b}(n)}{\tilde{a}(n+1)}\right) & \left(- \frac{1}{a(n+1)} +\frac{1}{\tilde{a}(n+1)}\right) \\
	a(n+1) - \tilde{a}(n+1) & 0
      \end{pmatrix}}_{=B(n)}     
      \begin{pmatrix} \tilde{u}(n) \\ \tilde{a}(n)\tilde{u}(n-1)\end{pmatrix}.
  \end{align*}
  Thus,
  \begin{align*}
    & \norm{\begin{pmatrix}
      u(n+1) - \tilde{u}(n+1)\\a(n+1)u(n) - \tilde{a}(n+1)\tilde{u}(n)
    \end{pmatrix}} \\
    & \leq \norm{A(n)} \norm{\begin{pmatrix}
      u(n) - \tilde{u}(n)\\a(n)u(n-1) - \tilde{a}(n)\tilde{u}(n-1)
    \end{pmatrix}} + \norm{B(n)}\norm{\begin{pmatrix} \tilde{u}(n) \\ \tilde{a}(n)\tilde{u}(n-1)\end{pmatrix}}.
   \end{align*}
   The Gronwall inequality in Lemma \ref{lem:Gronwall_discrete} yields the assertion for $n> 0$.
    For $n\leq 0$ we similarly obtain
    \begin{align*}
      & \begin{pmatrix}
      u(n) - \tilde{u}(n)\\a(n)u(n-1) - \tilde{a}(n)\tilde{u}(n-1)
    \end{pmatrix} \\
    & = \begin{pmatrix}
      u(n) - \tilde{u}(n) \\
      -a(n+1) u(n+1) + (z-b(n))u(n) + \tilde{a}(n+1) \tilde{u}(n+1) + (z-\tilde{b}(n))\tilde{u}(n) 
    \end{pmatrix} \\
    & = \begin{pmatrix}
      0 & \frac{1}{a(n+1)} \\
      -a(n+1) & \frac{z-b(n)}{a(n+1)}
    \end{pmatrix}  
    \begin{pmatrix}
      u(n+1) - \tilde{u}(n+1)\\a(n+1)u(n) - \tilde{a}(n+1)\tilde{u}(n)
    \end{pmatrix} \\&\quad
    + \begin{pmatrix} 
	 0 & \left( \frac{1}{a(n+1)} -\frac{1}{\tilde{a}(n+1)}\right) \\
	-a(n+1) + \tilde{a}(n+1) & \left(\frac{z-b(n)}{a(n+1)} - \frac{z-\tilde{b}(n)}{\tilde{a}(n+1)}\right)
      \end{pmatrix}    
      \begin{pmatrix} \tilde{u}(n+1) \\ \tilde{a}(n+1)\tilde{u}(n)\end{pmatrix} \\
      & = A(n)^{-1} \begin{pmatrix}
      u(n+1) - \tilde{u}(n+1)\\a(n+1)u(n) - \tilde{a}(n+1)\tilde{u}(n)
    \end{pmatrix} + B(n)^{-1} \begin{pmatrix} \tilde{u}(n+1) \\ \tilde{a}(n+1)\tilde{u}(n)\end{pmatrix}.
    \end{align*}
    Applying again the Gronwall inequality in Lemma \ref{lem:Gronwall_discrete} and taking into account Lemma \ref{lem:norm_unimodular} concludes the proof for $n\leq0$.
    \end{proof}

\begin{proof}[Proof of Theorem \ref{thm:Jacobi}]
  Let $z\in\C$ with $\abs{z}<\frac{e^{C/2}-\norm{a}_\infty}{\norm{\frac{1}{a}}_\infty}+\norm{b}_\infty-1$, and $u$ a solution of the difference equation
  \[a(n+1) u(n+1) + b(n) u(n) + a(n) u(n-1) - z u(n) = 0 \quad(n\in\Z).\]
  Assume that $u\neq 0$; then without loss of generality let
  \[\norm{\begin{pmatrix}
      u(1)\\a(1)u(0)
    \end{pmatrix}} = 1.\]
 Let $u_m$ be the solution of the corresponding difference equation, where $a$ and $b$ are replaced by the $p_m$-periodic versions $a_m$ and $b_m$ with $a_m(k+1) = a(k+1)$ and $b_m(k) = b(k)$ for $k\in\set{1,\ldots,p_m}$, respectively,
    satisfying
  \[\begin{pmatrix}
      u_m(1)\\a_m(1)u_m(0)
    \end{pmatrix} = \begin{pmatrix}
      u(1)\\a(1)u(0)
    \end{pmatrix} .\]
  
  Since for $n\in\Z$ we can estimate
  \[\norm{M_z(n)}_{1\to 1}, \norm{M_z(n)}_{\infty\to\infty} \leq \norm{\frac{1}{a}}_\infty \norm{z-b}_\infty + \norm{\frac{1}{a}}_\infty + \norm{a}_\infty,\]
  we obtain
  \begin{align*}
    \norm{T_z(n,0)} & \leq \sqrt{\norm{T_z(n,0)}_{1\to1}\norm{T_z(n,0)}_{\infty\to\infty}}\\
    & \leq \bigg(\norm{\frac{1}{a}}_\infty\norm{z-b}_\infty + \norm{\frac{1}{a}}_\infty + \norm{a}_\infty\bigg)^{\abs{n}}.
  \end{align*}
  Hence, it follows
  \[\norm{\begin{pmatrix}
      u_m(n+1)\\a_m(n+1)u_m(n)
    \end{pmatrix}}\leq \bigg(\norm{\frac{1}{a}}_\infty\norm{z-b}_\infty + \norm{\frac{1}{a}}_\infty + \norm{a}_\infty\bigg)^{\abs{n}} \quad (n\in\Z,m\in\N).\]  
  We now apply Lemma \ref{lem:Jacobi_diff} with $\tilde{a} = a_m$, $\tilde{b} = b_m$ and $\tilde{u} = u_m$.
  With the notation from this Lemma we obtain (noting that $A(n) = M_z(n)$)
  \[\norm{A(k)}\leq \norm{\frac{1}{a}}_\infty \norm{z-b}_\infty + \norm{\frac{1}{a}}_\infty + \norm{a}_\infty \quad(k\in\Z),\]
  and similarly
  \[\norm{B(k)} \leq \norm{\frac{1}{a}}_\infty^2\bigl(\abs{z}+\norm{b}_\infty + 1\bigr) + 1 \quad(k\in\Z).\]
  Thus, we have
  \begin{align*}
  & \norm{\begin{pmatrix}
      u(n+1) - u_m(n+1)\\a(n+1)u(n) - a_m(n+1)u_m(n)
    \end{pmatrix}}\\
    & \leq \bigg(\norm{\frac{1}{a}}_\infty\norm{z-b}_\infty + \norm{\frac{1}{a}}_\infty + \norm{a}_\infty\bigg)^{2\abs{n}}\bigg(\norm{\frac{1}{a}}_\infty^2\bigl(\abs{z}+\norm{b}_\infty + 1\bigr) + 1\bigg)\\
    & \qquad \cdot \sum_{k=\min\{n+1,1\}}^{\max\{0,n\}}\bigl(\abs{a(k+1)-a_m(k+1)} + \abs{b(k)-b_m(k)}\bigr) \\
    & \leq \bigg(\norm{\frac{1}{a}}_\infty\norm{z-b}_\infty + \norm{\frac{1}{a}}_\infty + \norm{a}_\infty\bigg)^{2\abs{n}}\bigg(\norm{\frac{1}{a}}_\infty^2\bigl(\abs{z}+\norm{b}_\infty + 1\bigr) + 1\bigg) \abs{n}\\
    & \qquad \cdot \max_{k=\min\{n+1,1\},\ldots,\max\{0,n\}}\bigl(\abs{a(k+1)-a_m(k+1)} + \abs{b(k)-b_m(k)}\bigr).
   \end{align*}
   Thus, for $n\in \{-p_m,\ldots,2p_m\}$ we obtain
  \begin{align}
  \nonumber
      & \norm{\begin{pmatrix}
      u(n+1) - u_m(n+1)\\a(n+1)u(n) - a_m(n+1)u_m(n)
    \end{pmatrix}}\\\nonumber 
    & \leq \bigg(\norm{\frac{1}{a}}_\infty\norm{z-b}_\infty + \norm{\frac{1}{a}}_\infty + \norm{a}_\infty\bigg)^{2p_m}\bigg(\norm{\frac{1}{a}}_\infty^2\bigl(\abs{z}+\norm{b}_\infty + 1\bigr) + 1\bigg)2p_m \\
    \label{eq:Jacobi_est}
    & \qquad \cdot \max_{k=-p_m+1,\ldots,2p_m}\bigl(\abs{a(k+1)-a_m(k+1)} + \abs{b(k)-b_m(k)}\bigr).
  \end{align}
  Since $a_m(k+1) = a(k+1)$ and $b_m(k) = b(k)$ for $k\in\set{1,\ldots,p_m}$, we infer
  \begin{align*}
    & \max_{k=-p_m+1,\ldots,2p_m}\bigl(\abs{a(k+1)-a_m(k+1)} + \abs{b(k)-b_m(k)}\bigr) \\
    & = \max_{k=-p_m+1,\ldots,p_m}\bigl(\abs{a(k+1)-a(k+1+p_m)} + \abs{b(k)-b(k+p_m)}\bigr).
  \end{align*}
  By assumption on $\abs{z}$ we obtain
  \[\bigg(\norm{\frac{1}{a}}_\infty\norm{z-b}_\infty + \norm{\frac{1}{a}}_\infty + \norm{a}_\infty\bigg)^{2p_m} < e^{Cp_m}\]
  for large $m$, so
  the right-hand side in \eqref{eq:Jacobi_est} tends to zero as $m\to \infty$. Hence, there exists $m_0\in\N$, such that for all $m\geq m_0$ we have
  \[\norm{\begin{pmatrix}
      u(n+1) - u_m(n+1)\\a(n+1)u(n) - a_m(n+1)u_m(n)
    \end{pmatrix}} \leq \frac{1}{4} \quad(n\in\{-p_m,\ldots,2p_m\}).\]
  Since
  \[\max\set{\norm{\begin{pmatrix}
		    u_m(n+1) \\ a_m(n+1) u_m(n)
                   \end{pmatrix}};\; n\in\set{-p_m,p_m,2p_m}} \geq \frac{1}{2}\]
  for all $m\in\N$ by Lemma \ref{lem:Jacobi_per}, we conclude that
  \[\limsup_{\abs{n}\to\infty} \norm{\begin{pmatrix}
      u(n+1)\\a(n+1)u(n)
    \end{pmatrix}} \geq \frac{1}{4}.\]
  Since $\inf_{n\in\Z}\abs{a(n)}>0$ this implies $\limsup_{\abs{n}\to\infty} 
\abs{u(n)} > 0$, and therefore $u\notin c_0(\Z)$.
\end{proof}

\begin{example}
  Let $\tilde{a},\tilde{b}\from\T\to\C$ be $\beta$-H\"older continuous and $\tilde{a}(x)\neq 0$ for all $x\in\T$. 
  Let $\alpha>0$ satisfy
  \[\abs{\alpha-\frac{p_m}{q_m}} \leq B m^{-q_m} \quad(m\in\N)\]
  for suitable $B>0$ and a suitable sequence $(\frac{p_m}{q_m})$ in $\Q$. Note that the set of all such numbers $\alpha$ is a dense $G_\delta$ set.
  Set $a(n):= \tilde{a}(\alpha\cdot n)$ and $b(n):= \tilde{b}(\alpha\cdot n)$ for all $n\in\N$.
  Then by the H\"older continuity and the assumption on $\alpha$ we observe
  \begin{align*}
    \abs{a(k+1) - a(k+1+q_m)} & \leq c\cdot \dist(\alpha q_m,\Z)^\beta \leq 
c\bigl(Bq_m m^{-q_m})^\beta \\
& = cB^\beta q_m^\beta e^{-\beta q_m \log m}
  \end{align*}
  for all $k\in\Z$ and $m\in\N$, and similarly for $b$.
  Thus, \eqref{eq:Jacobi_exp} is satisfied for all $C>0$, and therefore $H_{a,b}$ does not have any eigenvalues.
\end{example}

\section{The Sturm-Liouville case}
\label{sec:SL}

We say that
\[
  \mu \from \set{B\subseteq\R;\;B \text{ is a bounded Borel set}} \to \C
\]
is a \emph{local measure} if $\1_K\mu := \mu(\cdot\cap K)$ is a
complex Radon measure for any compact set $K\subseteq \R$.
Then there exist a (unique) nonnegative Radon measure $\nu$ on $\R$ and
a measurable function $\sigma\from\R\to \C$ such that $\abs{\sigma} = 1$ $\nu$-a.e.\ and
$\1_K\mu = \1_K\sigma\nu$ for all compact sets $K\subseteq \R$. The \emph{total variation} of $\mu$ is defined by $\abs{\mu}:=\nu$.
Let $\Mloc$ be the space of all local measures on $\R$.

A local measure $\mu\in \Mloc$ is called \emph{uniformly locally bounded} if
\[
  \norm{\mu}_\lu := \sup_{x\in\R} \abs{\mu}((x,x+1]) < \infty.
\]
Let $\M$ denote the space of all uniformly locally bounded local measures.
The space $\M$ naturally extends $L_{1,\loc,\unif}(\R)$ to measures.


Given 
$a\in L_\infty(\R)$ such that $\frac{1}{a}\in L_\infty(\R)$ and $\mu\in\M$ we 
consider the 
differential operator
\begin{align*}
	D(H_{a,\mu}) & := \set{u\in C_0(\R)\cap W_{1,\loc}^1(\R);\; -\partial 
a\partial u + u\mu \in C_0(\R)},\\
	H_{a,\mu} & := -\partial a\partial u + u\mu
\end{align*}
in $C_0(\R)$ (the space of continuous functions on $\R$ converging to $0$ at $\pm\infty$), where the terms are interpreted in the sense of distributions.

\begin{remark}
	We will also consider this operator in $L_p(\R)$, where $1\leq p<\infty$. 
Then $C_0(\R)$ has to be replaced by $L_p(\R)$.
\end{remark}

\begin{remark}
	For $\mu\in\M$ and $s,t\in\R$ we write
	\[\int_s^t\ldots\,d\mu := \begin{cases} 
	\int_{(s,t]} \ldots\,d\mu & s<t,\\
	0 & s=t,\\
	-\int_{(t,s]} \ldots\,d\mu & s>t.
	\end{cases}
	\]
\end{remark}

\begin{definition}
	Let $a\in L_\infty(\R)$ such that $\frac{1}{a}\in L_\infty(\R)$, $\mu \in 
\M$, $z\in\C$. We say that $u\in L_{1,\loc}(\R)$ is a solution of 
	\[H_{a,\mu}u = zu,\]
	if $u\in W_{1,\loc}^1(\R)$ and $-\partial a \partial u + u\mu = zu$ in the 
sense of distributions.
\end{definition}

\begin{lemma}
\label{lem:est_solutions}
  Let $a\in L_\infty(\R)$ such that $\frac{1}{a}\in L_\infty(\R)$, $\mu \in 
\M$, $z\in\C$, $u$ a solution of $H_{a,\mu} u = zu$, $I\subseteq\R$ an 
interval of length $1$, $1\leq p<\infty$. Then
  \[\norm{(au')|_I}_p\leq \norm{(au')|_I}_\infty\leq M\norm{u|_I}_\infty \leq 
(p+1)^{1/p}M^{(p+1)/p}\norm{u|_I}_p,\]
  where $M = (2\norm{a}_\infty+\norm{\mu-z\lambda}_{\unif})$. In particular, 
$u\in L_p(\R)$ implies $au'\in L_p(\R)$ and $u(x)\to 0$ as $\abs{x}\to\infty$
 implies $(au')(x\rlim)\to 0$ as $\abs{x}\to\infty$.
\end{lemma}

\begin{proof}
  The first estimate is trivial. For the second estimate note that
  \[\abs{\int_{I} u'(t\rlim)\, dt} \leq 2\norm{u|_I}_\infty,\]
  so there exists $t_0\in I$ such that $\abs{u'(t_0\rlim)}\leq 2\norm{u|_I}_\infty$.
  For $t\in I$ we compute
  \[\abs{(au')(t\rlim)}\leq \abs{(au')(t_0\rlim)} + \abs{\int_{t_0}^t u(s)\, d(\mu-z\lambda)(s)} \leq (2\norm{a}_{\infty}+\norm{\mu-z\lambda}_\unif)\norm{u|_I}_\infty\]
  which proves the second inequality.
  Now, choose $s_0\in I$ such that $\abs{u(s_0)} = \norm{u|_I}_\infty$. Then
  \[\abs{u(s_0+t)} = \abs{u(s_0) + \int_{s_0}^{s_0+t} u'(s)\,ds} \geq (1 - \abs{t} (2\norm{a}_\infty+\norm{\mu-z\lambda}_\unif))\norm{u|_I}_\infty\]
  for all $t\in\R$ with $s_0+t\in I$. Hence, we conclude
  \[\norm{u|_I}_p^p \geq \int_0^{1/M} (Mt)^p\norm{u|_I}_\infty^p\, dt = 
\frac{1}{(p+1)M}\norm{u|_I}_\infty^p. \qedhere\]
\end{proof}

\begin{remark}
  Let $a\in L_\infty(\R)$ such that $\frac{1}{a}\in L_\infty(\R)$, $\mu \in 
\M$. 
  There is also a unique realization of $-\partial a\partial + \mu$ via 
Sturm-Liouville theory, cf.\ \cite{EckhardtTeschl2011}.
  To this end, for $u\in W_{1,\loc}^1(\R)$ we define $A_{a,\mu} u\in 
L_{1,\loc}(\R)$ by
  \[(A_{a,\mu}u)(t) := (au')(t\rlim) - \int_0^t u(s)\, d\mu(s)\]
  for a.a.\ $t\in\R$. 
  Define
  \begin{align*}
    D(H_{a,\mu}^{\SL}) & := \set{u\in C_0(\R)\cap 
W_{1,\loc}^1(\R);\;A_{a,\mu}u\in W_{1,\loc}^1(\R),\, (A_{a,\mu}u)'\in 
C_0(\R)},\\
    H_{a,\mu}^{\SL}u & := -(A_{a,\mu}u)'.
  \end{align*}
\end{remark}

\begin{proposition}
\label{prop:Def_equal}
  Let $a\in L_\infty(\R)$ such that $\frac{1}{a}\in L_\infty(\R)$, $\mu \in 
\M$.  Then
  $H_{a,\mu} = H_{a,\mu}^{\SL}$.
\end{proposition}

\begin{proof}
  Let $u\in D(H_{a,\mu})$. Then $u\in W_{2}^1(\R)\subseteq W_{1,\loc}^1(\R)$ and $A_{a,\mu}u \in L_{1,\loc}(\R)$. Let $\varphi\in C_c^\infty(\R)$.
  By Fubini's Theorem we observe
  \begin{align*}
    \int (A_{a,\mu}u)\varphi' & = \int \bigg(au' - \int_0^xu(t)\,d\mu(t)\bigg)\varphi'(x)\, dx\\
    & = \int au'\varphi' - \int_\R \int_0^xu(t)\,d\mu(t)\varphi'(x)\, dx\\
    & = \int au'\varphi' + \int_{-\infty}^0  \!\int_{(-\infty,t)} \!\!\varphi'(x)\, dx u(t)\,d\mu(t) - \int_0^\infty \!\int_{[t,\infty)}\!\!\varphi'(x)\,dx u(t)\, d\mu(t)\\
    & = \int au'\varphi' + \int u\varphi\,d\mu
    = \int H_{a,\mu}u\varphi.
  \end{align*}
  Hence, $(A_{a,\mu}u)' = -H_{a,\mu}u \in C_0(\R)$. Therefore, also 
$A_{a,\mu}u\in W_{1,\loc}^1(\R)$ which implies $u\in D(H_{a,\mu}^{\SL})$, 
$H_{a,\mu} u = H_{a,\mu}^{\SL}u$.
  
  Conversely, let $u\in D(H_{a,\mu}^\SL)$.
  For $\varphi\in C_c^\infty(\R)$ 
  we compute by the help of Fubini's Theorem
  \begin{align*}
    \int -(A_{a,\mu}u)' \varphi & = \int (A_{a,\mu}u) \varphi'
    = \int au' \varphi' - \int \int_0^t u(s)\,d\mu(s) \varphi'(t)\,dt \\
    & = \int au' \varphi' + \int u\varphi\, d\mu. 
  \end{align*}
  Thus, $u\in D(H_{a,\mu})$, $H_{a,\mu}u = -(A_{a,\mu}u)'$.
\end{proof}

\begin{remark}
	Proposition \ref{prop:Def_equal} remains true if we consider the operators 
in $L_p(\R)$ with $1\leq p<\infty$. Furthermore, considering the operators 
in $L_2(\R)$, if additionally $a$ takes values only in a sector around the 
positive real axis, we obtain an equivalent characterization of 
$H_{a,\mu}$ via sectorial forms; cf.\ \cite[Remark 3.5 and Theorem 
3.6]{SeifertVogt2014} in case of Schr\"odinger operators.
\end{remark}

Since the operator is now defined, we will next focus on measuring distances of
 elements in $\M$.

\begin{definition}
For $\mu \in \M$ and a set $I \subseteq \R$ (which will usually be an interval) we define
\[
  \norm{\mu}_I := \sup\set{\Bigl|\int u\, d\mu\Bigr|;\;
  u\in W_{\!\infty}^1(\R),\: \spt u \subseteq I,\: \diam\spt u \leq 2,\: \norm{u'}_\infty \le 1 }.
\]
\end{definition}

For $\mu\in\M$ we define $\phi_\mu\from\R\to\C$ by
\[
  \phi_\mu(t) := \int_0^t d\mu
  = \begin{cases}
     \mu\bigl((0,t]\bigr) & \text{ if } t\ge0, \\
    -\mu\bigl((t,0]\bigr) & \text{ if } t<0.
    \end{cases}
\]

\begin{proposition}[see {\cite[Proposition 2.7, Remark 2.8 and Lemma 2.9]{SeifertVogt2014}}]
  Let $\mu\in\M$ and $x\in\R$. Then
  \[\norm{\mu}_{[x-1,x+1]} \leq \min_{c\in\C} \int_{x-1}^{x+1} \abs{\phi_\mu(t)-c}\,dt \leq 2\norm{\mu}_{[x-1,x+1]}.\]
  Hence, there exists $c_x=c_{\mu,x}\in\C$, such that 
  \[\int_{x-1}^{x+1} \abs{\phi_\mu(t)-c_x}\,dt \leq 2\norm{\mu}_{[x-1,x+1]}.\]
  Moreover, $c_{\mu,0}$ can be chosen such that $\abs{c_{\mu,0}}\leq \norm{\mu}_\unif$.
  Furthermore, for $\alpha,\beta\in\Z$, $\alpha\leq -1$, $\beta\geq1$ and $k\in\Z\cap[\alpha,\beta-1]$ we have
  \[\int_k^{k+1} \abs{\varphi_\mu(t)-c_0}\,dt\leq 2\max\{k+1,-k\}\norm{\mu}_{[\alpha,\beta]}.\]
\end{proposition}

We can now state Gordon's Theorem for Sturm-Liouville operators.

\begin{theorem}
\label{thm:SL}
  Let $a\in L_\infty(\R)$ such that $\frac{1}{a}\in L_\infty(\R)$ and $\mu \in 
\M$.
  Assume there exists $(p_m)$ in $(0,\infty)$ such that $p_m\to \infty$ and $C>0$ such that
  \begin{equation}
    \label{eq:SL_exp}
    e^{Cp_m} \left(\norm{a - a(\cdot-p_m)}_{L_1(-p_m,p_m)} + 
\norm{\mu-\mu(\cdot+p_m)}_{[-p_m,p_m]}\right)\to 0.
  \end{equation}
  Then $H_{(a,\mu)}$ does not have any eigenvalues with modulus less than $C^2\norm{\frac{1}{a}}_\infty - \norm{\mu}_\unif$.
\end{theorem}

As in the discrete case condition \eqref{eq:SL_exp} states that restrictions of $a$ and $\mu$ to the three pieces $[-p_m,0]$, $[0,p_m]$ and $[p_m,2p_m]$
do not differ to much (indeed, the difference tends to zero faster than a given exponential).

\begin{corollary}
  Let $a\in L_\infty(\R)$ such that $\frac{1}{a}\in L_\infty(\R)$ and $\mu \in 
\M$.
  Assume that \eqref{eq:SL_exp} is satisfied for all $C>0$. Then $H_{(a,\mu)}$ does not have any eigenvalues.
\end{corollary}

Taking into account the estimate in Lemma \ref{lem:est_solutions} we see that 
an $L_p(\R)$-eigenfunction for the $L_p(\R)$ operator is in fact an 
$C_0(\R)$-eigenfunction for the $C_0(\R)$ operator. Hence, we obtain the 
following corollary.

\begin{corollary}
	Let $1\leq p <\infty$, $a\in L_\infty(R)$ such that $\frac{1}{a}\in 
L_\infty(\R)$, $\mu\in\M$. Assume that \eqref{eq:SL_exp} holds true. Then 
$H_{a,\mu}$, considered as an operator in $L_p(\R)$, does not have any 
eigenvalues with modulus less than $C^2\norm{\frac{1}{a}}_\infty - 
\norm{\mu}_\unif$. In case \eqref{eq:SL_exp} holds true for all $C>0$, 
$H_{a,\mu}$ does not have any $L_p(\R)$-eigenvalues.
\end{corollary}

Again, we provide several lemmas for the proof of the theorem.

\begin{lemma}
  Let $a\in L_\infty(\R)$ such that $\frac{1}{a}\in L_\infty(\R)$, $\mu \in 
\M$, $z\in\C$.
  The following are equivalent:
  \begin{enumerate}
    \item
      $u$ is a solution of the equation $H_{a,\mu}u = zu$.
    \item
      For $s,t\in\R$ we have
      \[\begin{pmatrix}
	  u(t) \\
	  (au')(t\rlim)
	\end{pmatrix} = T_z(t,s)\begin{pmatrix}
	  u(s) \\
	  (au')(s\rlim)
	\end{pmatrix},\]
      where
      \[T_z(t,s) = \begin{pmatrix}
		    \uN(t;s) & \uD(t;s)\\
		    (a\uN'(\cdot;s))(t\rlim) & (a\uD'(\cdot;s))(t\rlim)
		    \end{pmatrix}\]
      and $\uN(\cdot;s)$, $\uD(\cdot;s)$ are the (Neumann and Dirichlet) solution(s) of $Hu = zu$ satisfying
      \begin{align*}
	\uN(s;s) & = 1 & \uD(s;s) & = 0 & \\
	(a\uN'(\cdot;s))(s\rlim) & = 0 & (a\uD'(\cdot;s))(s\rlim) & = 1 & \\
      \end{align*}
  \end{enumerate}
\end{lemma}

\begin{proof}
  ``(a)$\Rightarrow$(b)'': Fix $s,t\in\R$ and let
      \[T_z(t,s)\from \begin{pmatrix}
	  u(s) \\
	  (au')(s\rlim)
	\end{pmatrix}\mapsto \begin{pmatrix}
	  u(t) \\
	  (au')(t\rlim)
	\end{pmatrix}.\]
	Then $T_z(t,s)$ is linear and can be represented by a matrix, which we will also denote by $T_z(t,s)$.
	By the initial conditions for the Neumann and Dirichlet solution we observe
	\begin{align*}
	T_z(t,s) & = T_z(t,s)\begin{pmatrix} 1 & 0\\0 & 1\end{pmatrix} = T_z(t,s) \begin{pmatrix}
		    \uN(s;s) & \uD(s;s)\\
		    (a\uN'(\cdot;s))(s\rlim) & (a\uD'(\cdot;s))(s\rlim)
		    \end{pmatrix}\\
		  & = \begin{pmatrix}
		    \uN(t;s) & \uD(t;s)\\
		    (a\uN'(\cdot;s))(t\rlim) & (a\uD'(\cdot;s))(t\rlim)
		    \end{pmatrix}.
	\end{align*}
  ``(b)$\Rightarrow$(a)'': For $s,t\in\R$ we have
  \begin{align*}
    u(t) & = \uN(t;s)\cdot u(s) + \uD(t;s)\cdot (au')(s\rlim),\\
    (au')(t\rlim) & = (a\uN'(\cdot;s))(t\rlim)\cdot u(s) + (a\uD'(\cdot;s))(t\rlim)\cdot (au')(s\rlim).
  \end{align*}
  Differentiating the second equality, taking into account that $\uN(\cdot;s)$ and $\uD(\cdot;s)$ are solutions and noting the first equality yields
  \begin{align*}
   -(au')' & = -(a\uN'(\cdot;s))' \cdot u(s) - (a\uD'(\cdot;s))'\cdot (au')(s\rlim) \\
   & = z\bigl(\uN(\cdot;s)\cdot u(s) + \uD(\cdot;s)\cdot (au')(s\rlim)\bigr) \\
   & \quad - \bigl(\uN(\cdot;s)\cdot u(s) + \uD(\cdot;s)\cdot (au')(s\rlim)\bigr)\mu \\
   & = zu - u\mu.
  \end{align*}
  Hence, $u$ is a solution of $H_{a,\mu} u = zu$.	
\end{proof}

\begin{lemma}
\label{lem:SL_est1}
  Let $a\in L_\infty(\R)$ such that $\frac{1}{a}\in L_\infty(\R)$, $\mu \in 
\M$, $u$ a solution of $H_{a,\mu} u = 0$. Then
  \[\abs{u(t)} + \abs{(au')(t\rlim)} \leq \bigl(\abs{u(0)} + \abs{(au')(0\rlim)}\bigr)e^{(\norm{\frac{1}{a}}_\infty+\norm{\mu}_\unif)(\abs {t}+1)} \quad(t\in\R).\]
\end{lemma}

\begin{proof}
  Writing
  \begin{align*}
    u(t) & = u(0) + \int_0^t (au')(s\rlim)\frac{1}{a(s)}\,ds,\\
    (au')(t\rlim) & = (au')(0\rlim) + \int_0^t u(s)\, d\mu(s),
  \end{align*}
  we obtain for $\varphi(t):= \abs{u(t)} + \abs{(au')(t\rlim)}$ and $\nu:=\frac{1}{a}\lambda + \abs{\mu}$ the inequality
  \[\varphi(t)\leq \varphi(0) + \int_{(t,0]} \varphi(s)\, d\nu(s) \quad(t\leq 0).\]
  By Gronwall's inequality (see Lemma \ref{lem:Gronwall_cont}) we infer
  \[\varphi(t)\leq \varphi(0)e^{\nu((t,0])} \quad(t\leq 0).\]
  Since $\norm{\nu}_\unif\leq \norm{\frac{1}{a}}_\infty + \norm{\mu}_\unif$ and $\nu((t,0]) \leq \norm{\nu}_\unif(\abs{t}+1)$, we obtain the assertion for $t\leq 0$.
  
  For $t>0$ we set
  \[\varphi_-(s):= \abs{u(s)} + \abs{(au')(s\llim)} \leq \varphi(0) + \int_{(0,s)} \varphi_-(r)\,d\nu(r).\]
  The Gronwall's inequality in Lemma \ref{lem:Gronwall_cont} yields
  \[\abs{u(s)} + \abs{(au')(s\llim)} = \varphi_-(s) \leq \varphi(0)e^{\nu((0,s))} = \bigl(\abs{u(0)} + \abs{(au')(0\rlim)}\bigr)e^{\nu((0,s))}.\]
  For $s\downarrow t$ we the assertion follows, since $\nu((0,t])\leq \norm{\nu}_\unif(\abs{t}+1)$.
\end{proof}

\begin{lemma}
\label{lem:SL-diff1}
  Let $a,\tilde{a}\in L_\infty(\R)$ such that 
$\frac{1}{a},\frac{1}{\tilde{a}}\in L_\infty(\R)$, $\mu,\tilde{\mu}\in\M$, 
and $u$ and $\tilde{u}$ two solutions of
  $H_{a,\mu} u = 0$ and $H_{\tilde{a},\tilde{\mu}}\tilde{u} = 0$, respectively,
 satisfying
  \[\begin{pmatrix} u(0)\\(au')(0\rlim)\end{pmatrix} = \begin{pmatrix} \tilde{u}(0)\\(\tilde{a}\tilde{u}')(0\rlim)\end{pmatrix}.\]
  Then, for $s,t\in\R$ we have
  \begin{align*}
    & \begin{pmatrix} u(t)-\tilde{u}(t)\\(au')(t\rlim) - (\tilde{a}\tilde{u}')(t\rlim)\end{pmatrix} \\
    & = T_{\mu}(t,s) \begin{pmatrix} u(s)-\tilde{u}(s)\\(au')(s\rlim) - (\tilde{a}\tilde{u}')(s\rlim)\end{pmatrix} 
    + \int_s^t T_\mu(t,r)\begin{pmatrix}0 \\\tilde{u}(r)\end{pmatrix} \,d(\mu-\tilde{\mu})(r) \\
    & \quad + \int_s^t \bigg(\tfrac{1}{a(r)} - \tfrac{1}{\tilde{a}(r)}\bigg) T_\mu(t,r) \begin{pmatrix} (\tilde{a}\tilde{u}')(r\rlim)\\ 0\end{pmatrix}\,dr.
  \end{align*}  
\end{lemma}

\begin{proof}
  Without loss of generality, let $s=0$. Integrating by parts, we obtain
  \begin{align*}
    & \int_0^t T_\mu(r,0)^{-1}\begin{pmatrix} 0 \\ \tilde{u}(r)\end{pmatrix}\, d(\mu-\tilde{\mu})(r)
    = \begin{pmatrix} 
	- \int_0^t \uD(r)\tilde{u}(r)\,d(\mu-\tilde{\mu})(r) \\
       \int_0^t \uN(r)\tilde{u}(r)\,d(\mu-\tilde{\mu})(r)
      \end{pmatrix} \\
    & = \begin{pmatrix} u(0)\\(au')(0\rlim)\end{pmatrix} - T_\mu(0,t) \!\begin{pmatrix} \tilde{u}(t) \\(\tilde{a}\tilde{u}')(t\rlim)\end{pmatrix} - \int_0^t \bigl(\tfrac{1}{a(r)} \!-\! \tfrac{1}{\tilde{a}(r)}\bigr) T_\mu(0,r) \begin{pmatrix} (\tilde{a}\tilde{u}')(r\rlim)\\ 0\end{pmatrix}dr.
  \end{align*}
  Multiplying by $T_\mu(t,0)$ yields the assertion, since $T_\mu(t,0) T_\mu(r,0)^{-1} = T_\mu(t,r)$.
\end{proof}

\begin{lemma}
\label{lem:SL_diff2}
  Let $a,\tilde{a}\in L_\infty(\R)$ such that 
$\frac{1}{a},\frac{1}{\tilde{a}}\in L_\infty(\R)$, $\mu,\tilde{\mu}\in\M$, 
$c\in\C$ and $u$ and $\tilde{u}$ two solutions of
  $H_{a,\mu} u = 0$ and $H_{\tilde{a},\tilde{\mu}}\tilde{u} = 0$, respectively,
 satisfying
  \[\begin{pmatrix} u(0)\\(au')(0\rlim)\end{pmatrix} = \begin{pmatrix} \tilde{u}(0)\\(\tilde{a}\tilde{u}')(0\rlim) - c_{\mu-\tilde{\mu},0}u(0)\end{pmatrix}.\]
  Let $\alpha,\beta\in\Z$, $\alpha\leq -1$, $\beta\geq 1$. Let $c,\omega>0$ such that
  \[\abs{\uN(t,s)}, \abs{\partial_1 \uD(t\rlim,s)} \leq ce^{\omega\abs{t-s}} \quad(s,t\in\R).\]
  Then there exists a constant $C>0$ depending only on $\omega$ and $\norm{\tilde{\mu}}_\unif$ such that
  \[\abs{u(t) - \tilde{u}(t)} \leq Cce^{\omega \abs{t}} 
\norm{\tilde{u}|_{[\alpha,\beta]}}_\infty 
\bigl(\norm{(a-\tilde{a})|_{[\alpha,\beta]}}_{L_1(\R)}+\norm{\mu-\tilde{\mu}}_{
[\alpha,\beta]}\bigr) \quad(t\in[\alpha,\beta]).\]
\end{lemma}

\begin{proof}
  By Lemma \ref{lem:SL-diff1} we obtain
  \begin{align*}
    u(t) - \tilde{u}(t) & = -\uD(t) c_{\mu-\tilde{\mu},0} u(0) + \int_0^t \uD(t,r)\tilde{u}(r)\,d(\mu-\tilde{\mu})(r) \\
    & \quad + \int_0^t \bigg(\frac{1}{a(r)} - \frac{1}{\tilde{a}(r)}\bigg)\uN(t,r)(\tilde{a}\tilde{u}')(r\rlim)\,dr.
  \end{align*}
  Since $\uD(t,t) = 0$ we have
  \[\uD(t,r)\tilde{u}(r) = -\int_r^t \frac{d}{ds} \bigl(\uD(t,s)\tilde{u}(s)\bigr)\, ds.\]
  Fubini's Theorem then implies
  \begin{align*}
    u(t) - \tilde{u}(t) & = \int_0^t \bigl(c_{\mu-\tilde{\mu},0} - \varphi_{\mu-\tilde{\mu}}(s)\bigr)\frac{d}{ds} \bigl(\uD(t,s)\tilde{u}(s)\bigr)\, ds \\
    & \quad + \int_0^t \bigg(\frac{1}{a(r)} - \frac{1}{\tilde{a}(r)}\bigg)\uN(t,r)(\tilde{a}\tilde{u}')(r\rlim)\,dr.
  \end{align*}
  We now estimate $\frac{d}{ds} \bigl(\uD(t,s)\tilde{u}(s)\bigr)$. From $\uD(s,s) = 0$ and the assumed bound on $\abs{\partial_1 \uD(t\rlim,s)}$ we obtain $\abs{\uD(t,s)} \leq \frac{c}{\omega}e^{\omega\abs{t-s}}$ for all $s,t\in\R$.
  Furthermore, by Lemma \ref{lem:est_solutions} we have $\norm{(\tilde{a}\tilde{u}')|_{[\alpha,\beta]}}_{\infty} \leq (2\norm{a}_\infty + \norm{\tilde{\mu}}_\unif)\norm{\tilde{u}|_{[\alpha,\beta]}}_\infty$.
  Noting that $\uD(t,s) = -\uD(s,t)$ we hence obtain
  \begin{align*}
    \abs{\frac{\partial^+}{\partial s} \bigl(\uD(t,s)\tilde{u}(s)\bigr)} & = \abs{-\partial_1 \uD(s\rlim,t)\tilde{u}(s) + \uD(t,s)\tilde{u}'(s\rlim)} \\
    & \leq C_0ce^{\omega\abs{t-s}}\norm{\tilde{u}|_{[\alpha,\beta]}}_\infty \quad(s,t\in[\alpha,\beta]),
  \end{align*}
  with $C_0=\norm{\frac{1}{\tilde{a}}}_\infty\bigl(1+\frac{1}{\omega}(2\norm{a}_\infty + \norm{\tilde{\mu}}_\unif)\bigr)$.
  For $t\in[0,\beta]$ we therefore estimate
  \begin{align*}
    &\abs{u(t) - \tilde{u}(t)}\\
    & \leq  \int_0^t \abs{c_{\mu-\tilde{\mu},0} - \varphi_{\mu-\tilde{\mu}}(s)}\abs{\frac{d}{ds} \bigl(\uD(t,s)\tilde{u}(s)\bigr)} \, ds \\
    & \quad + \int_0^t \abs{\tfrac{1}{a(r)} - \tfrac{1}{\tilde{a}(r)}} \abs{\uN(t,r)(\tilde{a}\tilde{u}')(r\rlim)} \,dr \\
    & \leq C_0 c \norm{\tilde{u}|_{[\alpha,\beta]}}_\infty \sum_{k=1}^\beta \int_{k-1}^k e^{\omega(t-s)} \abs{c_{\mu-\tilde{\mu},0} - \varphi_{\mu-\tilde{\mu}}(s)}\, ds \\
    & \quad + \norm{\frac{1}{a}}_\infty\norm{\frac{1}{\tilde{a}}}_\infty 
(2\norm{\tilde{a}}_\infty + \norm{\tilde{\mu}}_\unif) 
\norm{\tilde{u}|_{[\alpha,\beta]}}_\infty c e^{\omega t} \int_0^t  
\norm{\tilde{a}(r)-a(r)}\, dr \\ 
    & \leq C_0 c \norm{\tilde{u}|_{[\alpha,\beta]}}_\infty \sum_{k=1}^\beta e^{\omega(t+1-k)} 2k\norm{\mu-\tilde{\mu}}_{[\alpha,\beta]} \\
    & \quad + \norm{\frac{1}{a}}_\infty\norm{\frac{1}{\tilde{a}}}_\infty 
(2\norm{\tilde{a}}_\infty + \norm{\tilde{\mu}}_\unif) 
\norm{\tilde{u}|_{[\alpha,\beta]}}_\infty c e^{\omega t} \int_0^t  
\norm{\tilde{a}(r)-a(r)}\, dr \\ 
    & \leq Cc\norm{\tilde{u}|_{[\alpha,\beta]}}_\infty e^{\omega t} 
\bigl(\norm{\mu-\tilde{\mu}}_{[\alpha,\beta]} + 
\norm{(a-\tilde{a})|_{[\alpha,\beta]}}_{L_1(\R)}\bigr)
  \end{align*}
  where $C = C_0 \sum_{k=1}^\infty 2ke^{-\omega(k+1)} + 
\norm{\frac{1}{a}}_\infty\norm{\frac{1}{\tilde{a}}}_\infty 
(2\norm{\tilde{a}}_\infty + \norm{\tilde{\mu}}_\unif)c$.
  The proof for the case $t\in[\alpha,0)$ is analogous.
\end{proof}

\begin{lemma}
\label{lemm:SL_est2}
    Let $a\in L_\infty(\R)$ such that $\frac{1}{a}\in L_\infty(\R)$, $\mu \in 
\M$, $u$ a solution of $H_{a,\mu} u = 0$, $\omega:= 
\bigl(\norm{\mu}_\unif\norm{\frac{1}{a}}_\infty^{-1}\bigr)^{1/2}$.
    Then
    \[\bigl(\omega^2 \abs{u(t)}^2 + \abs{(au')(t\rlim)}^2\bigr)^{1/2} \leq \bigl(\omega^2 \abs{u(0)}^2 + \abs{(au')(0\rlim)}^2\bigr)^{1/2} e^{\omega(\abs{t}+1/2)} \quad(t\in\R).\]
\end{lemma}

\begin{proof}
  Without loss of generality, let $\mu\neq 0$ (the case $\mu=0$ is trivial).
  
  (i) 
  We first assume that $\mu = \rho\lambda$ with a density $\rho\in C(\R)$. Then $au'\in C^1(\R)$ and $(au')' = \rho u$. Let $\phi(t):= \omega^2\abs{u(t)}^2 + \abs{(au')(t)}^2$. Then
  \begin{align*}
  	\abs{\phi'(t)} & = \abs{2\Re\bigl(\bigl(\tfrac{\omega}{\overline{a(t)}} + \abs{\rho(t)}\bigr)u(t) \overline{(au')(t)}\bigr)} \leq \bigl(\tfrac{\omega}{\abs{a(t)}} + \tfrac{\abs{\rho(t)}}{\omega}\bigr)\varphi(t).
  \end{align*}
  Hence, $\varphi(t)\leq \varphi(s)\exp(\omega\norm{\frac{1}{a}}_\infty\abs{t-s} + \frac{1}{\omega}\int_s^t\rho(r)\,dr)$ and therefore
  \[\bigl(\omega^2 \abs{u(t)}^2 + \abs{(au')(t\rlim)}^2\bigr)\leq \bigl(\omega^2 \abs{u(s)}^2 + \abs{(au')(s)}^2\bigr)e^{\omega\norm{\tfrac{1}{a}}_\infty\abs{t-s} + \frac{1}{\omega}\abs{\mu}([s,t])}\]
  for all $s,t\in\R$, $s<t$.
  
  (ii)
  By \cite[Proposition 2.5]{SeifertVogt2014} there exists $(\mu_n)$ in $\M$ such that $\mu_n$ has a smooth density and $\norm{\mu_n}_\unif\leq \norm{\mu}_\unif$ for all $n\in\N$, $\norm{\mu_n-\mu}_\R\to 0$ and $\limsup_{n\to\infty} \abs{\mu_n}(I) \leq \abs{\mu}(I)$ for all compact intervals $I\subseteq \R$.
  Then \cite[Lemma 2.4]{SeifertVogt2014} implies $\1_[\alpha,\beta]\mu_n\to \1_{[\alpha,\beta]}\mu$ weakly for all $\alpha,\beta\in\R$ such that $\mu(\{\alpha\}) = \mu(\{\beta\}) = 0$.
  
  (iii)
  For $n\in\N$ let $u_n$ be the solution of $H_{a,\mu_n}u_n = 0$ such that 
$u_n(0) = u(0)$, $(au_n)'(0\rlim) = (au')(0\rlim) + c_{\mu-\mu_n,0}u(0)$.
  By Lemma \ref{lem:SL_est1}, $(u_n)$ is uniformly bounded on any compact interval, so Lemma \ref{lem:SL_diff2} implies $u_n\to u$ locally uniformly. Hence, for $s,t\in\R$ with $\mu(\{s\}) = \mu(\{t\}) = 0$ we obtain
  \[(au_n')(t) - (au_n')(s) = \int_s^t u_n(r) \,d\mu_n(r) \to \int_s^t u(r)\,d\mu(r) = (au')(t\rlim) - (au')(s\rlim).\]
  By Lemma \ref{lem:est_solutions} also $(au_n')$ is uniformly bounded on $[0,1]$, so dividing by $a(s)$ and integration with respect to $s$ yields
  \begin{align*}
    (au_n')(t)\int_0^1 \frac{1}{a(s)}\,ds  - \bigl(u_n(1) - u_n(0)\bigr) & \to (au')(t\rlim)\int_0^1 \frac{1}{a(s)}\,ds  - \bigl(u(1) - u(0)\bigr),
  \end{align*}
  so $(au_n')(t) \to (au')(t\rlim)$.
  
  (iv)
  Let $t>s>0$ such that $\mu(\{s\}) = \mu(\{t\}) = 0$. By (i) we have
  \[\bigl(\omega^2 \abs{u_n(t)}^2 + \abs{(au_n')(t)}^2\bigr)\leq \bigl(\omega^2 \abs{u_n(s)}^2 + \abs{(au_n')(s)}^2\bigr)e^{\omega\norm{\tfrac{1}{a}}_\infty\abs{t-s} + \frac{1}{\omega}\abs{\mu_n}([s,t])}.\]
  Taking the limit $n\to\infty$ noting (ii) we obtain
  \[\bigl(\omega^2 \abs{u(t)}^2 + \abs{(au')(t\rlim)}^2\bigr)\leq \bigl(\omega^2 \abs{u(s)}^2 + \abs{(au')(s\rlim)}^2\bigr)e^{\omega\norm{\tfrac{1}{a}}_\infty\abs{t-s} + \frac{1}{\omega}\abs{\mu}([s,t])}.\]
  
  (v)
  For $t>0$ there exist sequences $s_n \in [0,t)$ and $(t_n)$ in $[t,\infty)$ 
such that $s_n\to 0$, $t_n\to t$ and $\mu(\{s_n\}) = \mu(\{t_n\}) = 0$ for 
all $n\in\N$.
  Thus, from (iv) we deduce
  \[\bigl(\omega^2 \abs{u(t)}^2 + \abs{(au')(t\rlim)}^2\bigr)\leq \bigl(\omega^2 \abs{u(0)}^2 + \abs{(au')(0\rlim)}^2\bigr)e^{\omega\norm{\tfrac{1}{a}}_\infty\abs{t} + \frac{1}{\omega}\abs{\mu}((0,t])}.\]
  Plugging in $\omega = 
\bigl(\norm{\mu}_\unif\norm{\frac{1}{a}}_\infty^{-1}\bigr)^{1/2}$ yields the 
assertion for $t\geq 0$. The case $t<0$ is proved analogously.
\end{proof}

%

We can proceed with the continuum version of Lemma \ref{lem:Jacobi_per}.
The proof is analogous to the discrete case in Lemma \ref{lem:Jacobi_per}, so we omit it here.

\begin{lemma}
\label{lem:SL_per}
  Let $a\in L_\infty(\R)$ such that $\frac{1}{a}\in L_\infty(\R)$, $\mu \in \M$
 and $p>0$ such that $a$ and $\mu$ are $p$-periodic. Let $z\in\C$ and $u$ a 
solution of $H_{a,\mu}u = zu$. Then
  \[\max\set{\norm{\begin{pmatrix} u(t)\\(au')(t\rlim)\end{pmatrix}};\; t\in\set{-p,0,2p}} \geq \frac{1}{2}\norm{\begin{pmatrix} u(0)\\(au')(0\rlim)\end{pmatrix}}.\]
\end{lemma}

%

\begin{lemma}[see {\cite[Lemma 5.1]{SeifertVogt2014}}]
\label{lem:equiv_Gordon}
  Let $\mu\in\M$, $C> 0$. Assume there exists $(p_m)$ in $(0,\infty)$ with $p_m\to \infty$ such that
  \[e^{Cp_m} \norm{\mu-\mu(\cdot+p_m)}_{[-p_m,p_m]}\to 0.\]
  Then there exists $(\mu_m)$ in $\M$ such that
      $\mu_m$ is periodic with period $p_m$ ($m\in\N$), and
      \[
        e^{Cp_m} \norm{\mu-\mu_m}_{[-p_m,2p_m]} \to 0 \qquad (m\to\infty).
      \]
  Moreover, the measures $\mu_m$ can be chosen such that
  \[
    \1_{[\alpha_m,p_m-\alpha_m]} \mu_m = \1_{[\alpha_m,p_m-\alpha_m]} \mu, \quad
    \norm{\mu_m}_\lu \le \bigl(1+\tfrac{1}{2\alpha_m}\bigr)\norm{\mu}_\lu
  \]
  for all $m\in\N$, with $0 < \alpha_m \le \frac{p_m}{2}$ and $\inf_{m\in\N} \alpha_m>0$.
\end{lemma}

\begin{proof}[Proof of Theorem \ref{thm:SL}]
  Without loss of generality, let $p_m\geq 4$ for all $m\in\N$. 
  Let $(\mu_m)$ and $(\alpha_m)$ as in Lemma \ref{lem:equiv_Gordon} such that $p_m+\alpha_m\in\N$ for all $m\in\N$, $\alpha_m\to \infty$ and $\frac{\alpha_m}{p_m}\to 0$.

  Assume that $z\in\C$ with $\abs{z}< C^2\norm{\frac{1}{a}}_\infty - \norm{\mu}_\unif$ is an eigenvalue of $H_{a,\mu}$.
  Let $u\neq 0$ be corresponding eigenfunction. Then $u\in C_0(\R)$.
  For $m\in\N$ let $u_m$ be the solution of $H_{a_m,\mu_m} u_m = zu_m$ satisfying $u_m(\alpha_m) = u(\alpha_m)$, $(a_m u_m')(\alpha_m\rlim) = (au')(\alpha_m')$.
  Then $u_m = u$ on $[\alpha_m,p_m-\alpha_m]$, since $\mu_m = \mu$ on this interval. Note that $c_{\mu-\tilde{\mu},\alpha_m+1} = 0$, since $\1_{[\alpha_m,\alpha_m+2]}(\mu_m-\mu) = 0$.
  By Lemma \ref{lem:SL_diff2}, for $t\in[-p_m,\alpha_m]$ we obtain
  \[\abs{u(t) - u_m(t)} \leq C_m 
e^{\omega_m\abs{t-(\alpha_m+1)}}\bigl(\norm{a-a_m}_{L_1(-p_m,\alpha_m+1)} + 
\norm{\mu-\mu_m}_{[-p_m,\alpha_m+1]}\bigr)\]
  where $\omega_m = \bigl(\norm{\mu_m-z\lambda}_\unif \norm{\frac{1}{a_m}}_\infty^{-1}\bigr)^{1/2}$ as in Lemma \ref{lemm:SL_est2}, and $C_m$ is only depending on $\omega_m$, $\norm{\frac{1}{a}}_\infty$, $\norm{\frac{1}{a_m}}_\infty$, $\norm{\mu}_\unif$ and $\norm{a}_\infty$, and similarly for $t\in [p_m-\alpha_m,2p_m]$.
  Hence,
  \begin{align}
  \label{eq:SL_conv}
    & \sup_{t\in[-p_m,2p_m]} \abs{u(t) - u_m(t)} \nonumber \\
    & \leq C_m e^{\omega_m (p_m+\alpha_m+1)}\bigl(\norm{a-a_m}_{L_1(-p_m,2p_m)}
 + \norm{\mu-\mu_m}_{[-p_m,2p_m]}\bigr).
  \end{align}
  Since $\norm{\frac{1}{a_m}}_\infty^{-1} \leq \norm{\frac{1}{a}}_\infty^{-1}$,
 we have
  \begin{align*}
    \omega_m^2 & = \norm{\frac{1}{a_m}}_\infty^{-1} \bigl(\norm{\mu_m}_\unif + 
\abs{z}\bigr) \leq \norm{\frac{1}{a}}_\infty^{-1} 
\bigl(\bigl(1+\tfrac{1}{2\alpha_m}\bigr)\norm{\mu}_\unif + \abs{z}\bigr) \\
    & \to \norm{\frac{1}{a}}_\infty^{-1} \bigl(\norm{\mu}_\unif + \abs{z}\bigr) < C^2,
  \end{align*} 
  so for large $m$ we obtain
  \[\omega_m(p_m+\alpha_m+1) \leq Cp_m.\]
  Thus, for $\varepsilon>0$ there exists $m_0\in\N$ such that such that $\abs{u(t)-u_m(t)}\leq \varepsilon$ for all $m\geq m_0$ and $t\in[-p_m,2p_m]$.
  Since $u\in C_0(\R)$, we there exists $m_1\geq m_0$ such that $\abs{u(t)}\leq \varepsilon$ for $\abs{t}\geq p_{m_1}-1=:t_1$. Then $\abs{u_m}\leq 2\varepsilon$ on $[-p_m,2p_m]\setminus (-t_1,t_1)$, for all $m\geq m_1$.
  By Lemma \ref{lem:est_solutions} we obtain $\abs{a_mu_m'}\leq 2\varepsilon(2\norm{a_m}_\infty + \norm{\mu_m-z\lambda}_\unif)$ on that set.
  Hence,
  \[\bigl(u_m(\pm p_m), (a_mu_m')(\pm p_m\rlim)\bigr),\bigl(u_m(2p_m), (a_mu_m')(2 p_m\rlim)\bigr) \to 0 \quad(m\to\infty).\]
  Lemma \ref{lem:SL_per} yields $\bigl(u_m(0), (a_mu_m')(0\rlim)\bigr)\to 0$. 
By Lemma \ref{lem:SL_est1} we now obtain $u_m\to 0$ locally uniformly. Since 
$u_m\to u$ locally uniformly by \eqref{eq:SL_conv}, we obtain $u=0$, a 
contradiction.
%
\end{proof}

\begin{remark}
  As shown in \cite[Remark 5.7 and Section 6]{SeifertVogt2014} for Schr\"odinger operators (i.e.\ $a=1$), the eigenvalue bound $C^2 - 
  \norm{\mu}_\unif$ can be sharpened to the optimal bound $C_\mu^2-\inf_{r>0} \norm{\mu}_{\unif,r}$, where
  \[C_\mu = -\liminf_{p\to\infty} \frac{1}{p} \ln \norm{\mu-\mu(\cdot+p)}_{[-p,p]}\]
  and
  \[\norm{\mu}_{\unif,r} := \frac{1}{r}\sup_{t\in\R} \abs{\mu}\bigl((a,a+r]\bigr).\]
  An analogous sharpening can also be done in our case, also yielding optimal bounds.
\end{remark}

\appendix

\section{Gronwall inequalities}

\begin{lemma}
\label{lem:Gronwall_discrete}
  Let $(x_n)_{n\in\N_0}$ in $[0,\infty)$, $x_0 = 0$, $(\alpha_n)_{n\in\N_0}$ in $[1,\infty)$ and $(\beta_n)_{n\in\N_0}$ in $[0,\infty)$ such that
  \[x_{n+1}\leq \alpha_n x_n + \beta_n \quad(n\in\N_0).\]
  Then
  \[x_{n} \leq \prod_{k=0}^{n-1} \alpha_k \sum_{k=0}^{n-1} \beta_k \quad(n\in\N_0).\]
\end{lemma}

\begin{proof}
  For $n=0$ the assertion ist tivial.
  For the induction step from $n$ to $n+1$ we compute
  \begin{align*}
    x_{n+1} & \leq \alpha_n x_n + \beta_n 
    \leq \alpha_n\bigg(\prod_{k=0}^{n-1} \alpha_k \sum_{k=0}^{n-1} \beta_k\bigg) \\
    & \leq \prod_{k=0}^{n} \alpha_k \sum_{k=0}^{n-1}\beta_k + \prod_{k=0}^{n} \alpha_k \beta_n
    = \prod_{k=0}^n \alpha_k \sum_{k=0}^n \beta_k. \qedhere
  \end{align*}
\end{proof}

\begin{lemma}
\label{lem:Gronwall_cont}
  Let $\alpha\from[0,\infty)\to [0,\infty)$ be measurable, $\mu$ a nonnegative Borel measure on $[0,\infty)$ and $u\in \mathcal{L}_{1,\loc}([0,\infty),\mu)$ such that
  \[u(t) \leq \alpha(t) + \int_{[0,t)} u(s)\,d\mu(s) \quad(t\geq 0).\]
  Then
  \[u(t) \leq \alpha(t) + \int_{[0,t)} \alpha(s)\exp\bigl(\mu\bigl((s,t)\bigr)\bigr)\,d\mu(s) \quad(t\geq 0).\]
\end{lemma}

\begin{proof}
  (i) Iterating the inequality yields
  \[u(t)\leq \alpha(t) + \int_{[0,t)} \alpha(s) \sum_{k=0}^{n-1}\mu^{\otimes k}\bigl(A_k(s,t)\bigr) \, d\mu(s) + R_n(t) \quad(n\in\N, t\geq 0),\]
  where
  \[R_n(t) := \int_{[0,t)} u(s) \mu^{\otimes n}\bigl(A_n(s,t)\bigr) \, d\mu(s)\]
  is the remainder, 
  \[A_k(s,t) := \set{(s_1,\ldots,s_k)\in(s,t)^k}{s_1<\ldots<s_k}\]
  is an $k$-dimensional simplex and
  \[\mu^{\otimes 0}\bigl(A_0(s,t)\bigr) := 1.\]
  
  (ii) Let $0\leq s<t$. We now prove
  \[\mu^{\otimes k}\bigl(A_k(s,t)\bigr) \leq \frac{\mu\bigl((s,t)\bigr)^k}{k!} \quad(k\in\N_0).\]
  Indeed, let $S_k$ be the set of all permutations of $\set{1,\ldots,k}$. For $\sigma\in S_k$ let
  \[A_{k,\sigma}(s,t) := \set{(s_1,\ldots,s_k)\in(s,t)^k}{s_{\sigma(1)}<\ldots<s_{\sigma(k)}}.\]
  Then for $\sigma\neq \sigma'$ we obtain $A_{k\sigma}(s,t)\cap A_{k,\sigma'}(s,t) = \varnothing$. 
  Furthermore,
  \[\bigcup_{\sigma\in S_k} A_{k,\sigma}(s,t) \subseteq (s,t)^k.\]
  Hence,
  \[k! \mu^{\otimes k} \bigl(A_k(s,t)\bigr) = \sum_{\sigma\in S_k}\mu^{\otimes k} \bigl(A_k(s,t)\bigr) \leq \mu^{\otimes k}\bigl((s,t)^k\bigr) = \mu\bigl((s,t)\bigr)^k.\]
  
  (iii) By (ii), we obtain
  \[\abs{R_n(t)} \leq \frac{\mu\bigl((s,t)\bigr)^n}{n!} \int_{[0,t)} \abs{u(s)}\,d\mu(s) \quad(n\in\N, t\geq 0).\]
  Since $u$ is locally integrable with respect to $\mu$ we obtain $R_n\to 0$ pointwise.
  Thus, (i) yields
  \begin{align*}
    u(t) & \leq \alpha(t) + \int_{[0,t)} \alpha(s) \sum_{k=0}^{n-1}\frac{\mu\bigl((s,t)\bigr)^k}{k!} \, d\mu(s) + R_n(t) \\
    & \leq \alpha(t) + \int_{[0,t)} \alpha(s) \exp\bigl(\mu\bigl((s,t)\bigr)\bigr) \, d\mu(s) + R_n(t) \\
    & \to \alpha(t) + \int_{[0,t)} \alpha(s) \exp\bigl(\mu\bigl((s,t)\bigr)\bigr) \, d\mu(s). \qedhere
  \end{align*}
\end{proof}

\section{Unimodular Matrices}

An $n\times n$-matrix $A$ with complex entries is called unimodular, if $\det A = 1$. Let $SL(n,\C)$ be the set of all unimodular $n\times n$-matrices.

\begin{lemma}
\label{lem:norm_unimodular}
  For $A\in SL(2,\C)$ we have
  \[\norm{A} = \norm{A^{-1}}.\]
\end{lemma}

\begin{proof}
  By the Schur decomposition we may assume that
  \[A = \begin{pmatrix} a & b \\ 0 & \frac{1}{a}\end{pmatrix},\]
  with $a,b\in\C$.
  Then 
  \[A^{-1} = \begin{pmatrix} \frac{1}{a} & -b \\ 0 & a\end{pmatrix}.\]
  We compute
  \[A^TA = \begin{pmatrix} a^2 & ab \\ ab & b^2 + \frac{1}{a^2}\end{pmatrix},\qquad (A^{-1})^T A^{-1} = \begin{pmatrix} \frac{1}{a^2} & -\frac{b}{a} \\ -\frac{b}{a} & a^2+ b^2 \end{pmatrix}.\]
  Since these two matrices have the same traces and determinants, their characteristic polynomials are equal and therefore they have the same eigenvalues. Hence, $\norm{A} = \norm{A^{-1}}$.
\end{proof}

\bigskip

\noindent
Christian Seifert\\
Technische Universit\"at Hamburg-Harburg\\
Institut f\"ur Mathematik \\
Schwarzenbergstra{\ss}e 95 E \\
21073 Hamburg, Germany \\
{\tt christian.se\rlap{\textcolor{white}{hugo@egon}}ifert@tuhh\rlap{\textcolor{white}{darmstadt}}.de}

\end{document}